\documentclass[letterpaper]{amsart}

\makeatletter
\def\@defaultbiblabelstyle#1{[#1]}
\makeatother

\makeatletter
\def\@setauthors{
  \begingroup
  \def\thanks{\protect\thanks@warning}
  \trivlist
  \centering\footnotesize \@topsep30\p@\relax
  \advance\@topsep by -\baselineskip
  \item\relax
  \author@andify\authors
  \def\\{\protect\linebreak}
  \authors
  \ifx\@empty\contribs
  \else
    ,\penalty-3 \space \@setcontribs
    \@closetoccontribs
  \fi
  \endtrivlist
  \endgroup
}
\def\@settitle{\begin{center}
  \baselineskip14\p@\relax
    \bfseries
  \@title
  \end{center}
}
\makeatother

\usepackage{amsxtra}
\usepackage{mathtools}
\mathtoolsset{showonlyrefs}
\usepackage{extarrows}
\usepackage{color}
\usepackage{xcolor}
\usepackage{enumitem}
\usepackage{tikz}
\usetikzlibrary{calc}
\usepackage{subcaption}
\setlist[enumerate]{label=\upshape(\arabic*)}
\allowdisplaybreaks[4]
\usepackage[colorlinks=true, citecolor=blue]{hyperref}
\usepackage{amsthm}
\usepackage{amssymb}
\usepackage{amsmath}
\usepackage{mathrsfs}
\usepackage{upref}
\numberwithin{equation}{section}

\newtheorem{theorem}{Theorem}[section]
\newtheorem{lemma}[theorem]{Lemma}
\newtheorem{proposition}[theorem]{Proposition}
\newtheorem{corollary}[theorem]{Corollary}

\newtheorem{question}[theorem]{Question}

\theoremstyle{definition}

\theoremstyle{remark}
\newtheorem{remark}[theorem]{Remark}

\begin{document}

\title{Strongly dominant weight polytopes are cubes}



\author{Gaston Burrull}
\address{(Gaston Burrull) \newline \indent Beijing International Center for Mathematical Research, Peking University, No.\ 5 Yiheyuan Road, Haidian District, Beijing 100871, China}
\email{gaston(at)bicmr(dot)pku(dot)edu(dot)cn}

\author{Tao Gui}
\address{(Tao Gui) \newline \indent Beijing International Center for Mathematical Research, Peking University, No.\ 5 Yiheyuan Road, Haidian District, Beijing 100871, China}
\email{guitao(at)amss(dot)ac(dot)cn}

\author{Hongsheng Hu}
\address{(Hongsheng Hu) \newline \indent School of Mathematics, Hunan University, Changsha 410082, China}
\email{huhongsheng(at)amss(dot)ac(dot)cn}

\begin{abstract}
For any root system of rank $r$, we study the ``dominant weight polytope'' $P^\lambda$ associated with a strongly dominant weight $\lambda$. We prove that $P^\lambda$ is combinatorially equivalent to the $r$-dimensional cube. As an application, we give a new proof of the known formulas for Betti numbers of Peterson varieties in classical Lie types.
\end{abstract}

\maketitle

\section{Introduction}

A \emph{weight polytope} is the convex hull of the weights that occur in some highest weight representation of a Lie algebra \cite{fulton2013representation,humphreys1972}. Weight polytopes are a recurrent object of study in representation theory of Lie algebras (see, for example, \cite{besson2024weight,gao2018generating,li2014cross,Panyushev97,Walton21}). In type $A$, a weight polytope is precisely the classical permutohedron $P_n\left(x_1, \ldots, x_n\right)$, that is, the convex hull of the $n!$ points obtained from $\left(x_1, \ldots, x_n\right)$ by permuting the coordinates. The face structure of weight polytopes was studied in \cite{maxwell1989wythoff} and \cite{renner2009descent}, while Postnikov computed the volume and the number of lattice points of weight polytopes in the seminal paper \cite{postnikov2009permutohedra}.

In this paper, we study the ``dominant weight polytopes'' for any root system. Let $\Phi$ be a root system of 
rank $r$ in the sense of \cite{humphreys1972} (that is, $\Phi$ is crystallographic, reduced, and finite) and let $E$ be the $r$-dimensional Euclidean space where $\Phi$ lives. 
The space $E$ is called \emph{weight space}, and points in $E$ are called \emph{weights}. 
Let $(-|-)\colon E \times E \to \mathbb{R}$ be the inner product on $E$.
For any root $\alpha \in \Phi$, we denote the corresponding coroot $\frac{2 \alpha}{(\alpha|\alpha)}$ by $\alpha\spcheck$. Then $\Phi\spcheck := \left\{\alpha\spcheck \in E \mid \alpha \in \Phi \right\}$ is the dual root system of $\Phi$.
We fix a set $\Delta = \left\{\alpha_i \mid i = 1,\dots, r\right\}$ of simple roots of $\Phi$.
Let
\begin{equation*}
    s_i\colon E \to E, \quad x \mapsto x - (x|\alpha_i \spcheck) \alpha_i, \quad i = 1, \dots, r,
\end{equation*}
be the simple reflections. The Weyl group $W_f:= \langle s_1 ,\dots, s_r \rangle$  acts on $E$ and preserves $(-|-)$. The \emph{dominant Weyl chamber} is the open cone in $E$ given by
\begin{equation*}
    C_+ := \{x \in E \mid (x|\alpha_i) > 0, \text{ for all } i = 1, \dots, r\}.
\end{equation*} 
A point $\lambda \in E$ is called \emph{strongly dominant} if $\lambda \in C_+$, and \emph{dominant} if $\lambda \in \overline{C_+}$, the closure of $C_+$. 
Given $\lambda \in \overline{C_+}$, we define the \emph{dominant weight polytope $P^\lambda$ associated with $\lambda$} by
\begin{equation*}
       P^\lambda := \operatorname{Conv} (W_f \lambda) \cap \overline{C_+} \subset E, 
\end{equation*}
where $\operatorname{Conv} (W_f \lambda)$ is the convex hull of the finite set $W_f \lambda$ in $E$. That is, $P^\lambda$ is the convex polytope obtained from intersecting the cone $\overline{C_+}$ with the convex set $\operatorname{Conv} (W_f \lambda)$. See Figure \ref{fig:A3points} for examples.

\begin{figure}[ht]
    \centering
    \begin{subfigure}{0.4\textwidth}
    \centering
    \includegraphics{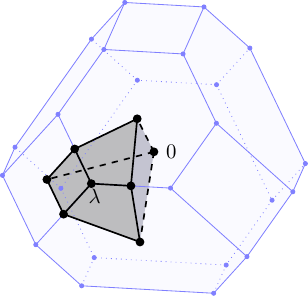}
        \caption{$\Phi$ of type $A_3$ and $\lambda=5\alpha_1+7\alpha_2+6\alpha_3.$}
        \label{fig:poly576A3}
    \end{subfigure}
    \qquad
    \begin{subfigure}{0.4\textwidth}
    \centering
    \includegraphics{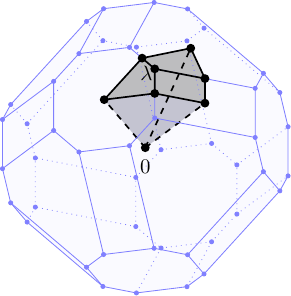}
        \caption{$\Phi$ of type $B_3$ and  $\lambda=5\alpha_1+7\alpha_2+8\alpha_3.$ }
        \label{fig:poly578B3}
    \end{subfigure}
    \caption{In each figure, $\lambda\in C_+$, and the blue polyhedron is $\operatorname{Conv} (W_f \lambda)$.
    The gray polyhedron is $P^{\lambda}$ and it is combinatorially equivalent to the $3$-dimensional cube. }
    \label{fig:A3points}
\end{figure}

As far as we know, the dominant weight polytope appeared implicitly for the first time in the seminal paper of Guillemin and Sternberg \cite{guillemin1982convexity}, and its applications have increased during recent years.
It plays an important role in actions and representations of reductive groups \cite{kaveh2015crystal, kk10, kaveh2012convex}, in the representation of complex simple Lie algebras \cite{boysal2024kostants}, and in the description of cohomology rings of spherical varieties \cite{kaveh11}. 
To investigate moduli stacks of pointed chains of $\mathbb{P}^1$ related to the Losev--Manin moduli spaces, Blume \cite{blume2015toric} studied the toric orbifold associated with (the normal fan of) a dominant weight polytope in types $A$, $B$, and $C$.
A special case of Horiguchi--Masuda--Shareshian--Song's results \cite{horiguchi2021toric} 
says that the rational cohomology ring of the toric orbifold associated with a dominant weight polytope in classical Lie types is isomorphic to the cohomology of the corresponding famous Peterson variety\footnote{Originally introduced by Peterson to describe the quantum cohomology rings of partial flag varieties \cite{peterson1997quantum}.}, see also \cite{abe2023peterson}.
Moreover, if $\lambda$ is in the coroot lattice $\mathbb{Z}\Phi^\vee$, then $P^\lambda$ is closely related to the ``dominant lower Bruhat interval'' corresponding to the translation by $\lambda$ and its asymptotic behavior in the affine Weyl group $\mathbb{Z}\Phi^\vee\rtimes W_f$ \cite{BGH-log-conc}. 

In combinatorics of convex polytopes, it is crucial to determine the face structure (or combinatorial type) of a polytope. 
However, it is usually difficult to do it. For example, very little is known about the face structure of a $d$-polytope for $d \geq 3$ \cite[p.\ 38--39]{grunbaum03}. 
In this paper, we prove that if $\lambda$ is strongly dominant, the face structure of $P^\lambda$ is equivalent to that of the $r$-dimensional cube (Theorem \ref{cube}). 
Topologically, this is equivalent to the existence of a (piecewise linear) homeomorphism between $P^\lambda$ and the cube, which restricts to homeomorphisms between their facets (and hence all the faces), see \cite[Section 2.2]{zbMATH00722614}.
In particular, the face structure of $P^\lambda$ when $\lambda\in C_+$ only depends on the rank $r$ of $\Phi$. 
Its $2^r$ vertices are in canonical one-to-one correspondence with subsets of $\Delta$. 
Furthermore, each vertex of $P^\lambda$ can be computed explicitly from the Cartan matrix of $\Phi$ (see Corollary \ref{cor-vertices}). 
We generalize these results to a bigger family of polytopes with almost identical proofs (Theorem \ref{thm-gen}).

After submitting the first preprint version of this work, Marc Besson communicated to us the existence of \cite{besson2021vertices}.
In their paper, a description of the vertices of $P^\lambda$, when $\lambda$ is dominant and integral, was given, which is similar to Remark \ref{rem-wall}.
Their work focuses on the extremal rays of the Kostka cone $\mathscr{K}:= \{(\lambda, \mu) \mid \lambda, \mu \text{ dominant, } \mu \in P^\lambda\}$ and the vertices of its slices, these slices are the $P^\lambda$'s (they call them intersection polytopes). See also Theorem 1.2 in \cite{boysal2024kostants}.
It is important to emphasize that the results in the present paper do not follow from the enumeration of the vertices nor the $f$-vector of $P^{\lambda}$ (see Remark \ref{non-cube}).

We want to conclude this introduction with a question. As we will see in Remark \ref{rem-wall}, if $\lambda$ lies on some walls of $C_+$, the structure of the polytope $P^\lambda$ is more complicated and interesting. See Figure \ref{fig:A3pointswall} for examples.

\begin{figure}[ht]
    \centering
    \begin{subfigure}{0.4\textwidth}
    \centering
    \includegraphics{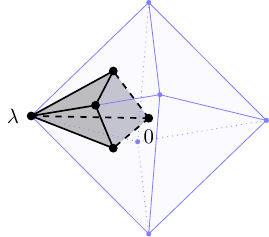}
        \caption{$\Phi$ of type $A_3$ and $\lambda=\alpha_1+2\alpha_2+\alpha_3.$}
        \label{fig:poly121A3}
    \end{subfigure}
    \qquad
    \begin{subfigure}{0.4\textwidth}
    \centering
    \includegraphics{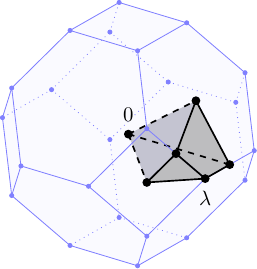}
        \caption{$\Phi$ of type $B_3$ and  $\lambda=2\alpha_1+3\alpha_2+3\alpha_3.$ }
        \label{fig:poly233B3}
    \end{subfigure}
    \caption{In each figure, $\lambda\in \overline{C_+}\setminus C_+$, and the blue polyhedron is $\operatorname{Conv} (W_f \lambda)$. 
    The gray polyhedron is $P^{\lambda}$. In this case, it has fewer vertices than the $3$-dimensional cube. }
    \label{fig:A3pointswall}
\end{figure}

\begin{question}
What is the face structure of $P^{\lambda}$ for $\lambda \in \overline{C_{+}} \backslash C_{+}$? How much does it depend on the root system $\Phi$ and the weight $\lambda$?
\end{question}

The remainder of the paper is organized as follows. In Section \ref{sec-main}, we study the dominant weight polytope and prove our main theorem (Theorem \ref{cube}). In Section \ref{sec-generalization}, we generalize these results to a broader family of polytopes. 
In Section \ref{Sec-Application}, we apply our results to obtain Betti numbers of Peterson varieties.

\subsection*{Acknowledgments}

The authors would like to thank Marc Besson, Michel Brion, Shiliang Gao, Yibo Gao, Mikiya Masuda, Connor Simpson, and Haozhi Zeng for valuable discussions.

\section{The dominant weight polytope} \label{sec-main}

For general notions and facts about convex polytopes, we refer to the standard textbooks \cite{grunbaum03, zbMATH00722614}.

We fix a dominant weight $\lambda \in \overline{C_+}$, and denote by $[r]$ the finite set $\{1, \dots, r\}$.
Firstly, we have the fact that the dominant weight polytope $P^\lambda$ is the intersection of two closed simplicial cones.

\begin{proposition}[{\cite[Proposition 4.2]{BGH-log-conc}}] \label{prop-cone}
We have $P^\lambda = \overline{C_+} \cap \overline{Q^\lambda}$, where $\overline{Q^\lambda}$ is the closure of the open simplicial cone 
\begin{equation*}
  Q^\lambda := \Bigl\{\lambda - \sum_{i \in [r]} c_i \alpha_i\spcheck \Bigm\vert c_i \in \mathbb{R}_{> 0} \Bigr\}.
\end{equation*}
\end{proposition}

\begin{proof}
    If $\mu \in \overline{C_+}$, then $\mu \in \operatorname{Conv}(W_f \lambda)$ if and only if $\lambda - \mu$ is a non-negative linear combination of simple coroots $\{\alpha_1\spcheck, \dots, \alpha_r\spcheck\}$ (as well as simple roots), see \cite[Proposition 8.44]{hall2015gtm222}.
    Therefore, $P^\lambda = \operatorname{Conv} (W_f \lambda) \cap \overline{C_+} = \overline{C_+} \cap \overline{Q^\lambda}$.
\end{proof}

Let $\varpi_1\spcheck, \dots, \varpi_r\spcheck \in E$ be the basis of fundamental coweights, which is dual to the basis of simple roots. In formulas, $(\varpi_i\spcheck | \alpha_j) = \delta_{i,j}$ where $\delta_{i,j}$ is the Kronecker delta.
For any subsets $I,J \subseteq [r]$, we define
\begin{align*}
  C_I & := \Bigl\{0 + \sum_{i\in I} a_i \varpi_i\spcheck \Bigm| a_i \in \mathbb{R}_{>0}\Bigr\}, \\
  Q^\lambda_J & := \Bigl\{\lambda - \sum_{i \in J} c_i \alpha_i\spcheck \Bigm| c_i \in \mathbb{R}_{>0} \Bigr\}, \\
  F_{I,J} & := C_I \cap Q^\lambda_J \text{ (possibly empty)}.
\end{align*}
The $F_{I,J}$'s depend on $\lambda$, but we omit it in the notation.
It is easy to see that $\bigl\{\overline{C_I} \bigm| I \subseteq [r] \bigr\}$ is the set of faces of $\overline{C_+}$, which contains $2^r$ elements.
Each $C_I$ has dimension $\lvert  I \rvert$ and is open in its closure.
In particular, $C_\emptyset = \{0\}$, $C_{[r]} = C_+$, and $\overline{C_+} = \bigsqcup_{I \subseteq [r]} C_I$ is the disjoint union of all the $C_I$'s.
Similarly, $\bigl\{\overline{Q^\lambda_J} \bigm| J \subseteq [r] \bigr\}$ is the set of faces of $\overline{Q^\lambda}$.
In particular, $Q^\lambda_\emptyset = \{\lambda\}$, and $Q^\lambda_{[r]} = Q^\lambda$.
Therefore, we have the following decomposition.

\begin{lemma} \label{lem-Plambda=union}
  The polytope $P^\lambda = \bigsqcup_{I,J \subseteq [r]} F_{I,J}$ is the disjoint union of all the $F_{I,J}$'s.
\end{lemma}

\begin{proof}
  $P^\lambda = \overline{C_+} \cap \overline{Q^\lambda} = \Bigl(\bigsqcup_{I \subseteq [r]} C_I \Bigr) \cap \Bigl(\bigsqcup_{J \subseteq [r]} Q^\lambda_J\Bigr) = \bigsqcup_{I, J \subseteq [r]} \bigl(C_I \cap Q^\lambda_J\bigr)= \bigsqcup_{I, J \subseteq [r]} F_{I,J}$.
\end{proof}

Note that $F_{I,J}$ is convex since $C_I$ and $Q^\lambda_J$ are convex.
It follows that the closure $\overline{F_{I,J}}$ is also convex.

\begin{lemma} \label{lem-points-between-points}
  Suppose $I_1, \dots, I_n, J_1, \dots, J_n \subseteq [r]$ are subsets of $[r]$, and $F_{I_k,J_k} \ne \emptyset$ for all  $k$. 
  Let $x_k \in F_{I_k,J_k}$ for $k=1, \dots, n$, and let $r_1, \dots, r_n \in \mathbb{R}_{>0}$ be arbitrary positive numbers such that $\sum_{k=1}^n r_k = 1$. Then 
\begin{equation*}
\sum_{k=1}^n r_k x_k \in F_{(\bigcup_k I_k), (\bigcup_k J_k)}.
\end{equation*}
  In particular, $F_{(\bigcup_k I_k), (\bigcup_k J_k)} \ne \emptyset$.
\end{lemma}

\begin{proof}
  For each $1\leq k \leq n$, we write
  \begin{equation*}
    x_k = \sum_{i \in I_k} a_{k,i} \varpi_i \spcheck = \lambda - \sum_{ j \in J_k} c_{k,j} \alpha_j\spcheck
  \end{equation*}
  where $a_{k,i}, c_{k,j} \in \mathbb{R}_{>0}$.
  Then we have
  \begin{alignat*}{3}
    & \sum_{1 \le k \le n} r_k x_k & & = \sum_{1 \le k \le n} \sum_{i \in I_k} r_k a_{k,i} \varpi_i \spcheck & & = \lambda - \sum_{1 \le k \le n} \sum_{j \in J_k} r_k c_{k,j} \alpha_j \spcheck \\
    & & & = \sum_{i \in \bigcup_k I_k} a_i \varpi_i \spcheck & & = \lambda - \sum_{j \in \bigcup_k J_k} c_j \alpha_j \spcheck
  \end{alignat*}
  for some $a_i, c_j \in \mathbb{R}_{>0}$. This concludes the proof.
\end{proof}
The point $\sum_{1 \le k \le n} r_k x_k$ from Lemma \ref{lem-points-between-points} lies in the convex hull of the points $x_1, \dots, x_n$.

Before digging into further details about the sets $F_{I,J}$, we need to introduce some notations.
Let $M = \bigl((\alpha_i|\alpha_j\spcheck)\bigr)_{i,j\in[r]}$ be the Cartan matrix of $\Phi$.
For any subset $J \subseteq [r]$, we denote by $M_J$ the submatrix $\bigl((\alpha_{i}|\alpha_{j}\spcheck)\bigr)_{i,j\in J}$ of $M$. In particular, $M_{[r]}=M$.
The matrix $M_J$ is the Cartan matrix of the root subsystem of $\Phi$ corresponding to $J$, which is invertible. For $\gamma \in \overline{C_+}$, we denote 
    \[p^\gamma_J := \gamma - \sum_{j \in J} c^\gamma_j \alpha_j\spcheck, \text{ where } (c^\gamma_j)_{j \in J} := M_J^{-1} \cdot \bigl( (\alpha_i | \gamma) \bigr)_{i \in J}.\]
Clearly, $(\alpha_j|p^\gamma_J) = 0$ for every $j \in J$. The following lemma is immediate.
\begin{lemma}\label{lem-x-in-CI}
Let $I\subset [r]$. Then $x\in C_I$ if and only if $(\alpha_i|x)>0$ for all $i\in I$ and $(\alpha_j|x)=0$ for all $j\notin I$.
\end{lemma}

Henceforth, we will focus on the case where $\lambda$ is strongly dominant.

\begin{lemma} \label{lem-single-point}
  Suppose $\lambda \in C_+$, and $I, J \subseteq [r]$, such that $[r] = I \sqcup J$ is a disjoint union.
  Then $F_{I,J} = \{p^\lambda_J\}$ is a singleton.
\end{lemma}

\begin{proof}
  If $F_{I,J} \ne \emptyset$ and  $x \in F_{I,J}$, then 
  \begin{equation*}
    x = \lambda - \sum_{j \in J} c_j \alpha_j\spcheck, \text{ where } c_j \in \mathbb{R}_{>0}.
  \end{equation*}
  By Lemma \ref{lem-x-in-CI}, for every $i\in J$, we have
  \begin{equation*}
    \sum_{j \in J} c_j (\alpha_i |\alpha_j\spcheck)= (\alpha_i | \lambda).
  \end{equation*}
  In other words, the vector $(c_j)_{j \in J}$ satisfies $M_J \cdot (c_j)_{j\in J} = \bigl( (\alpha_i | \lambda) \bigr)_{i \in J}$. Since $M_J$ is invertible, this implies $c_j=c^\lambda_j$ for all $j\in J$, that is, $x=p^\lambda_J$. 
  Therefore, $F_{I,J}$ has at most one single point.
  
  Conversely, we need to prove that $p^\lambda_J\in F_{I,J}$. 
  Since $\lambda$ is strongly dominant, we have $(\alpha_j | \lambda) > 0$ for all $j \in J$. 
  The inverse of any Cartan matrix has non-negative entries\footnote{A general statement holds: 
  If the off-diagonal entries of a positive-definite symmetric matrix $A$ are non-positive, then $A^{-1}$ has non-negative entries.}
  (see \cite[Exercise 13.8]{humphreys1972} or \cite[Reference Chapter, Section 2]{ov1990lie} for a list of all possible cases). 
  In particular, $M_J^{-1}$ is entry-wise non-negative.
  Therefore, $c^\lambda_j > 0$ for all $j \in J$, so $p^\lambda_J \in Q^\lambda_J$.
  For $i\in I$ and $j\in J$, we have $i\neq j$ so then $(\alpha_i|\alpha_j^\vee)\leq 0$. Therefore, for every $i\in I$, we have
  \begin{equation*}
    (\alpha_i|p^\lambda_J) = (\alpha_i | \lambda) - \sum_{j \in J} c^\lambda_j (\alpha_i | \alpha_j\spcheck) \geq (\alpha_i | \lambda) >0.
  \end{equation*}
  Recall that $(\alpha_j|p^\lambda_J) = 0$ for all $j \in J$. By Lemma \ref{lem-x-in-CI}, we have $p^\lambda_J\in C_I$.
\end{proof}

The following lemma determines when the set $F_{I,J}$ is nonempty.

\begin{lemma} \label{lem-FIJ-nonempty}
  Suppose $\lambda \in C_+$.
  For two subsets $I,J \subseteq [r]$, $F_{I,J} \ne \emptyset$ if and only if $I \cup J = [r]$.
\end{lemma}

\begin{proof}
  Suppose $x \in F_{I,J} \ne \emptyset$. As before, we write 
  \begin{equation*}
    x = \lambda - \sum_{j \in J} c_j \alpha_j\spcheck, \text{ where } c_j \in \mathbb{R}_{>0}.
  \end{equation*} 
  If $i \notin J$, we have
  \begin{equation*}
    (x | \alpha_i) = (\lambda | \alpha_i) - \sum_{j \in J} c_j (\alpha_j\spcheck | \alpha_k) \ge (\lambda | \alpha_i) > 0.
  \end{equation*}
  By Lemma \ref{lem-x-in-CI}, we have $i\in I$. Therefore $I \cup J = [r]$.

  Conversely, suppose $I \cup J = [r]$. Let $I_0 = I \setminus (I \cap J)$ and $J_0 = J \setminus (I\cap J)$.
  Then $[r] = I_0 \sqcup (I\cap J) \sqcup J_0$ is a disjoint union.
  By Lemma \ref{lem-single-point}, both of the sets $F_{I_0,J}$ and $F_{I,J_0}$ have a single point.
  Then by Lemma \ref{lem-points-between-points}, we have $F_{I,J} = F_{(I_0 \cup I), (J \cup J_0)} \ne \emptyset$.
\end{proof}

The following proposition describes the closure relation of the $F_{I,J}$'s.
\begin{proposition} \label{prop-face-closure-Plambda}
  Suppose $\lambda \in C_+$.
  For two subsets $I, J \subseteq [r]$, we have 
\begin{equation*}
\overline{F_{I,J}} = \bigsqcup_{\substack{I^\prime \subseteq I\\ J^\prime \subseteq J}} F_{I^\prime, J^\prime}.
\end{equation*}
  In particular, if $I, I^\prime, J, J^\prime \subseteq [r]$, then $F_{I^\prime, J^\prime} \subseteq \overline{F_{I,J}} $ if and only if $I^\prime \subseteq I$ and $J^\prime \subseteq J$.
\end{proposition}

\begin{proof}
  Notice that for any $I, J \subseteq [r]$, we have,
  \begin{align*}
    \overline{C_I} & = \Bigl\{\sum_{i \in I} a_i \varpi_i\spcheck \Bigm| a_i \in \mathbb{R}_{\ge 0} \Bigr\} = \bigsqcup_{I^\prime \subseteq I} C_{I^\prime}, \\
    \overline{Q^\lambda_J} & = \Bigl\{ \lambda - \sum_{j \in J} c_j \alpha_j\spcheck \Bigm| c_j \in \mathbb{R}_{\ge 0} \Bigr\} = \bigsqcup_{J^\prime \subseteq J} Q^\lambda_{J^\prime}.
  \end{align*}
  So 
  \begin{equation*}
    \overline{F_{I,J}} = \overline{C_I \cap Q^\lambda_J} \subseteq \overline{C_I} \cap \overline{Q^\lambda_J} = \bigsqcup_{\substack{I^\prime \subseteq I\\ J^\prime\subseteq J}} F_{I^\prime,J^\prime}.
  \end{equation*}
  This proves ``$\subseteq$''.

  Conversely, let $x \in F_{I^\prime, J^\prime}$. 
  By Lemma \ref{lem-FIJ-nonempty}, we have $I^\prime \cup J^\prime = [r]$. This implies $I \cup J = [r]$. By Lemma \ref{lem-FIJ-nonempty} again, we have $F_{I,J} \ne \emptyset$. We choose an arbitrary point $y \in F_{I,J}$.
  By Lemma \ref{lem-points-between-points}, we have $ax + (1-a) y \in F_{I,J}$ for $0 < a < 1$.
  Therefore,
  \begin{equation*}
    x = \lim_{a \to 1^-} \bigl(ax + (1-a)y\bigr) \in \overline{F_{I, J}}.
  \end{equation*}
  This proves ``$\supseteq$''.
\end{proof}

For $1\leq k \leq r$, we define the following affine hyperplanes in $E$:
\begin{align*}
  H_k & := \Bigl\{\sum_{i \in [r] \setminus \{k\}} a_i \varpi_i\spcheck \Bigm| a_i \in \mathbb{R} \Bigr\}, \\
  H_k^\lambda & := \Bigl\{\lambda - \sum_{j \in [r] \setminus \{k\}} c_j \alpha_j\spcheck \Bigm| c_j \in \mathbb{R} \Bigr\}.
\end{align*}
We also define the open half-spaces (whose closures are called closed half-spaces) in $E$:
\begin{align*}
  H_{k,+} & := \Bigl\{\sum_{1 \le i \le r} a_i \varpi_i\spcheck \Bigm| a_i \in \mathbb{R}\text{ for all } i, \text{ and } a_k > 0 \Bigr\}, \\
  H_{k,+}^\lambda & := \Bigl\{\lambda - \sum_{1 \le j \le r} c_j \alpha_j\spcheck \Bigm| c_j \in \mathbb{R}\text{ for all } j, \text{ and } c_k > 0 \Bigr\}.
\end{align*}
Clearly, we have $\overline{C_+} = \bigcap_{1 \le k \le r} \overline{H_{k,+}}$ and $\overline{Q^\lambda} = \bigcap_{1 \le k \le r} \overline{H_{k,+}^\lambda}$, and thus
\begin{equation*}
  P^\lambda = \overline{C_+} \cap \overline{Q^\lambda} = \Bigl(\bigcap_{1 \le k \le r} \overline{H_{k,+}} \Bigr) \cap \Bigl(\bigcap_{1 \le k \le r} \overline{H_{k,+}^\lambda}\Bigr)
\end{equation*}
is the intersection of the closed half-spaces.

We have the following description of the faces of $P^\lambda$.

\begin{theorem} \label{thm-all-faces-of-Plambda}
Let $\lambda \in C_+$ be strongly dominant.
Then, the set $\{\overline{F_{I,J}} \mid I, J \subseteq [r], I \cup J = [r]\}$ is the set of all faces of $P^\lambda$. Furthermore, if $I \cup J  = [r]$, then $F_{I,J}$ is a real manifold of dimension $\lvert I \rvert + \lvert J \rvert - r$.
\end{theorem}

\begin{proof}
For any two points $x,y \in P^\lambda$, we write $x \sim y$ if for every $k = 1, \dots, r$, either $x,y \in H_k$ or $x,y \in H_{k,+}$, and either $x,y \in H^\lambda_k$ or $x,y \in H^\lambda_{k,+}$.
Clearly, ``$\sim$'' is an equivalence relation on $P^\lambda$. A face of $P^\lambda$ is nothing but the closure of an equivalence class.

  By Lemmas \ref{lem-Plambda=union} and \ref{lem-FIJ-nonempty}, we have
  \begin{equation*}
    P^\lambda = \bigsqcup_{I,J \subseteq [r]} F_{I,J} = \bigsqcup_{I,J \subseteq [r], I \cup J = [r]} F_{I,J}.
  \end{equation*}
  The first part of the theorem follows from the fact that for two points $x \in F_{I,J}$ and $y \in F_{I^\prime, J^\prime}$, we have $x \sim y$ if and only if $I = I^\prime$ and $J = J^\prime$.

  For the second part of the theorem, consider the $\lvert I \rvert$-dimensional affine subspace
  \begin{equation*}
    X := \bigcap_{k \in [r] \setminus I} H_k = \Bigl\{\sum_{i \in I} a_i \varpi_i\spcheck \Bigm| a_i \in \mathbb{R} \text{ for all } i \in I \Bigr\},
  \end{equation*}
  and the $\lvert J \rvert$-dimensional  affine subspace
  \begin{equation*}
    Y := \bigcap_{j \in [r] \setminus J} H_j^\lambda = \Bigl\{\lambda - \sum_{j \in J} c_j \alpha_j\spcheck  \Bigm| c_j \in \mathbb{R} \text{ for all } j \in J \Bigr\}.
  \end{equation*}
  Since $I \cup J = [r]$, by Lemma \ref{lem-FIJ-nonempty}, the intersection $X \cap Y$ is a nonempty affine subspace.
  The vectors $\{\varpi_i \spcheck \mid i\in I\}$ are parallel to $X$ and the vectors $\{\alpha_j\spcheck \mid j \in J \setminus (I \cap J)\}$ are parallel to $Y$. Furthermore, the union $\{\varpi_i \spcheck, \alpha_j\spcheck \mid i \in I, j \in J \setminus (I \cap J)\}$ is a basis of $E$.
  Then, $X$ intersects with $Y$ transversally, and hence $\dim (X \cap Y) = \dim X + \dim Y - r = \lvert I \rvert + \lvert J \rvert - r$.
  Notice that $C_I$ is an open cone in $X$, and $Q^\lambda_J$ is an open cone in $Y$.  
  Therefore, the intersection $F_{I,J} = C_I \cap Q^\lambda_J$ (which is nonempty) is an open subset in the affine subspace $X \cap Y$, and hence a real manifold of dimension $\lvert I \rvert + \lvert J \rvert - r$.
\end{proof}

For a polytope $P$, there is a partial order $\le$ on the set of faces of $P$ defined by $F \le F^\prime$ if $F \subseteq F^\prime$.
A polytope $P$ of dimension $r$ is called \emph{combinatorially equivalent} to the $r$-dimensional cube if the set of faces of $P$ is isomorphic to the set of faces of the standard cube $[0,1]^{\times r} \subset \mathbb{R}^r$ as partially ordered sets.

\begin{theorem} \label{cube}
Let $\lambda \in C_+$ be strongly dominant.
The polytope $P^\lambda$ is combinatorially equivalent to the $r$-dimensional cube.
\end{theorem}

\begin{proof}
  A face $H$ of the $r$-dimensional cube $[0,1]^{\times r}$ in the Euclidean space $\mathbb{R}^r$ is uniquely determined by a partition $[r] = I_0 \sqcup I_1 \sqcup I_{01}$ in the following way
  \begin{equation*}
      H = H(I_0,I_1,I_{01}) := \left\{ (x_1,\dots, x_r) \in \mathbb{R}^r \middle| 
       \begin{gathered}
           \text{$x_i = 0$ if  $i \in I_0$,} \\
           \text{$x_j =1$ if  $j \in I_1$,} \\
           \text{$x_k \in [0,1]$ if $k \in I_{01}$}
       \end{gathered}\right\}.
  \end{equation*}
  Suppose $H^\prime = H^\prime(I^\prime_0, I^\prime_1, I^\prime_{01})$ is another face of the cube corresponding to the partition $[r] = I_0^\prime \sqcup I_1^\prime \sqcup I_{01}^\prime$.
  Then, $H \subseteq H^\prime$ if and only if $I_0^\prime \subseteq I_0$ and $I_1^\prime \subseteq I_1$.

  We define a map $\theta$ from the set of faces of the cube $[0,1]^{\times r}$ to the set of faces of the polytope $P^\lambda$ by
  \begin{equation*}
    \theta\colon H(I_0,I_1,I_{01}) \mapsto \overline{F_{(I_0 \cup I_{01}), (I_1 \cup I_{01})}}.
  \end{equation*}
  Then, by Theorem \ref{thm-all-faces-of-Plambda}, this map is well defined and is a bijection with inverse
  \begin{equation*}
    \theta^{-1}\colon \overline{F_{I,J}} \mapsto H \bigl(I \setminus (I \cap J), J \setminus (I \cap J), I\cap J \bigr).
  \end{equation*}
  Therefore, to show that $P^\lambda$ is combinatorially equivalent to the cube, it suffices to show that for any two faces $H$ and $H^\prime$ of the cube, 
  \[\text{$H \subseteq H^\prime$ if and only if $\theta(H) \subseteq \theta(H^\prime)$.}\]
  By Proposition \ref{prop-face-closure-Plambda}, this is equivalent to saying that for any two partitions $[r] = I_0 \sqcup I_1 \sqcup I_{01} = I_0^\prime \sqcup I_1^\prime \sqcup I_{01}^\prime$ we have the following equivalence,
  \begin{equation*} 
    \text{$I_0^\prime \subseteq I_0$ and $I_1^\prime \subseteq I_1$  if and only if $I_0 \cup I_{01} \subseteq I_0^\prime \cup I_{01}^\prime$ and $I_1 \cup I_{01} \subseteq I_1^\prime \cup I_{01}^\prime$.}
  \end{equation*}
  But this is clear since $I_0, I_1, I_0'$ and $I_1'$ are the complements of $I_1 \cup I_{01}$, $I_0 \cup I_{01}$, $I_1' \cup I_{01}'$ and $I_0' \cup I_{01}'$, respectively.
\end{proof}

\begin{remark}\label{non-cube}
    The intersection of two closed simplicial cones might be a bounded $r$-dimensional polytope with the same $f$-vector as the $r$-dimensional cube---that is, for every $i$, it contains precisely $\binom{r}{i}2^{r-i}$ faces of dimension $i$---and yet be not combinatorially equivalent to the cube. For instance, Figure \ref{figure-not-cube} shows a polyhedron, which is the intersection of two simplicial cones in $\mathbb{R}^3$. It consists of $8$ vertices, $12$ edges, and $6$ faces, but it is not combinatorially equivalent to the cube since it contains one face with $3$ vertices and another face with $5$ vertices.
\end{remark}
    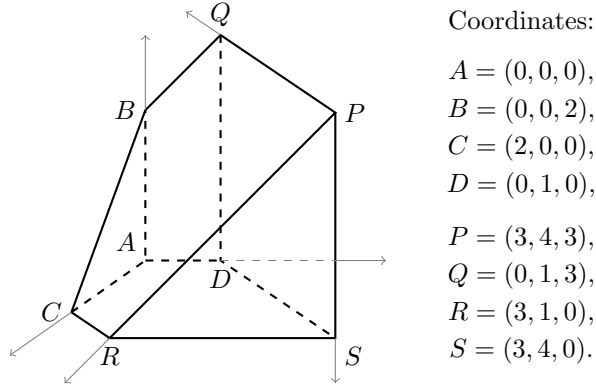
\begin{figure}[ht]
      \centering
      \begin{tikzpicture}
        \coordinate [label=150:$A$] (A) at (0,0);
        \coordinate [label=left:$B$] (B) at (0,2);
        \coordinate [label=left:$C$] (C) at (215:1.2);
        \coordinate [label=below:$D$] (D) at (1,0);

        \coordinate [label=right:$P$] (P) at ($(0,3) +  (215:1.8) + (4,0)$);
        \coordinate [label=above:$Q$] (Q) at (1,3);
        \coordinate [label=below:$R$] (R) at ($(215:1.8) + (1,0)$);
        \coordinate [label=315:$S$] (S) at ($(R) + (3,0)$);
        
        \draw [dashed, thick] (C) -- (A) -- (B);
        \draw [dashed, thick] (Q) -- (D) -- (S);
        \draw [dashed, thick] (A) -- (D);
        \draw [thick] (R) -- (C) -- (B) -- (Q) -- (P) -- (R) -- (S) -- (P);
        \draw [->, gray] (B) -- (0,3);
        \draw [->, gray] (C) -- (215:2.2);
        \draw [gray, dashed] (D) -- (2.5,0);
        \draw [->, gray] (2.5,0) -- (3.2,0);
        \draw [->, gray] (Q) -- ($1.3*(Q) - 0.3*(P)$);
        \draw [->, gray] (R) -- ($1.2*(R) - 0.2*(P)$);
        \draw [->, gray] (S) -- ($1.2*(S) - 0.2*(P)$);

        \node at (5,3.2) {Coordinates:};
        \node at (5,2.5) {$A = (0,0,0)$,};
        \node at (5,2) {$B = (0,0,2)$,};
        \node at (5,1.5) {$C = (2,0,0)$,};
        \node at (5,1) {$D = (0,1,0)$,};
        \node at (5,0.3) {$P = (3,4,3)$,};
        \node at (5,-0.2) {$Q = (0,1,3)$,};
        \node at (5,-0.7) {$R = (3,1,0)$,};
        \node at (5,-1.2) {$S = (3,4,0)$.};
      \end{tikzpicture}
      \caption{Let $\mathcal{C}_1$ be the simplicial cone with vertex $A$ and extremal rays $AB$, $AC$, $AD$, and let $\mathcal{C}_2$ be the simplicial cone with vertex $P$ and extremal rays $PQ$, $PR$, $PS$. The intersection $\mathcal{C}_1 \cap \mathcal{C}_2$ is a convex polytope with $8$ vertices ($A,B,C,D,P,Q,R,S$), $12$ edges and $6$ faces. However, this is not combinatorially equivalent to a cube since it has some triangles and pentagons as faces.}\label{figure-not-cube}
    \end{figure}

Note that for any partition $[r] = I \sqcup J$, the point in $F_{I,J}$ is a $0$-dimensional face of $P^\lambda$, in other words, a vertex.
As a consequence, we have the following minimal vertex representation of $P^\lambda$, and we can explicitly compute its $2^r$ vertices using the Cartan matrix of the root system as demonstrated in the proof of Lemma \ref{lem-single-point}.

\begin{corollary} \label{cor-vertices}
Let $\lambda \in C_+$ be strongly dominant. 
The polytope $P^\lambda$ is the convex hull of the $2^r$ vertices 
    \begin{equation*}
        \bigl\{F_{I,J} \bigm| [r] = I \sqcup J\bigr\}=\{p^\lambda_J \mid J \subseteq [r] \}.
    \end{equation*}
\end{corollary}

\begin{remark} \label{rem-wall}
The identity $P^\lambda = \operatorname{Conv} \{p^\lambda_J \mid J \subseteq [r]\}$ from Corollary \ref{cor-vertices}
still holds for $\lambda\in\overline{C_+}$. 
This is because $P^\lambda$ and the $p^\lambda_J$'s depend continuously on $\lambda$. 
However, if $\lambda\notin C_+$, we have $p^\lambda_J = p^\lambda_{J^\prime}$ for some subsets $J,J^\prime \subseteq [r]$. In particular, if $\lambda\notin C_+$, then $P^\lambda$ is not combinatorially equivalent to the $r$-dimensional cube.
\end{remark}

Recall that the \emph{Minkowski sum} of two subsets $A$ and $B$ of a Euclidean space is defined by $A+B:=\{\boldsymbol{a}+\boldsymbol{b} \mid \boldsymbol{a} \in A, \boldsymbol{b} \in B\}$.
If $A$ and $B$ are convex polytopes, then so is $A+B$. The following proposition relates different dominant weight polytopes. Although this is briefly mentioned in \cite{besson2021vertices}, we provide an elementary proof.

\begin{proposition} \label{Prop-Minkowski}
    Let $\lambda, \mu \in \overline{C_+}$ be dominant.
    Then, $P^{\lambda+\mu}$ is the Minkowski sum of $P^\lambda$ and $P^\mu$.
\end{proposition}

\begin{proof}
    By Proposition \ref{prop-cone}, it is not hard to see that the Minkowski sum of $P^\lambda$ and $P^\mu$ satisfies the inequalities
    \begin{equation*}
    (x|\alpha_i) \geq 0, \text{ and } 
    (x|\varpi_i) \leq (\lambda|\varpi_i)+(\mu|\varpi_i), \text{ for all } i = 1, \dots, r,
    \end{equation*}
    which are the same inequalities defining $P^{\lambda+\mu}$.
    Thus, $P^\lambda + P^\mu \subseteq P^{\lambda+\mu}$.

    Conversely, by Remark \ref{rem-wall}, 
    we have $P^{\gamma} = \operatorname{Conv}\{p^{\gamma}_J \mid J \subseteq [r]\}$ for $\gamma \in \{\lambda, \mu\}$.
    Moreover, we have $p^{\lambda+\mu}_J=p^\lambda_J + p^\mu_J \in P^\lambda + P^\mu$.
    Since the Minkowski sum of two convex sets is also convex, we have $P^{\lambda + \mu} = \operatorname{Conv}\{p^{\lambda + \mu}_J \mid J \subseteq [r]\} \subseteq P^\lambda + P^\mu$. 
\end{proof}

\section{A generalization of the dominant weight polytope} \label{sec-generalization}

The same arguments apply to the following generalization of the results above.

\begin{theorem} \label{thm-gen}
    Let $\{v_1, \dots, v_r\}$ be a basis for an $r$-dimensional Euclidean space $E$ such that $(v_i | v_j) \le 0$ for distinct $i, j$. 
    Let $\lambda \in \overline{C_+}$ be a point, where
    \[C_+ := \{x \in E \mid (x | v_i) > 0 \text{ for all } i \in [r]\}.\]
    We define $P^\lambda :=\overline{C_+} \cap \overline{Q^\lambda}$, where 
    \[Q^\lambda := \Bigl\{\lambda - \sum_{i \in [r]} c_i v_i \Bigm | c_i > 0 \text{ for all } i \Bigr\}.\]
    Then, we have the following statements:
    \begin{enumerate}
        \item The set $P^\lambda$ is a bounded convex polytope.
        \item If $\lambda \in C_+$, then $P^\lambda$ is combinatorially equivalent to the $r$-dimensional cube, whose $2^r$ vertices are
        \[\Bigl\{ \lambda - \sum_{j \in J} c^\lambda_j v_j \Bigm | J \subseteq [r] \Bigr\},\]
        where $(c^\lambda_j)_{j \in J} := M_J^{-1} \cdot \bigl( (v_i | \lambda) \bigr)_{i \in J}$ and $M_J := \bigl((v_i | v_j)\bigr)_{i,j \in J}$ is the submatrix of the Gram matrix of the basis $\{v_1, \dots, v_r\}$.
        \item If $\lambda, \mu \in \overline{C_+}$, then $P^{\lambda+ \mu}$ is the Minkowski sum of $P^\lambda$ and $P^\mu$.
    \end{enumerate}
\end{theorem}

\begin{proof}
    Here, we only prove the boundedness of $P^\lambda$. 
    The proofs of the other assertions are the same as in Section \ref{sec-main}.

    Let $\{u_1, \dots, u_r\}$ be the basis of $E$ dual to $\{v_1, \dots, v_r\}$, that is, $(u_i | v_j) = \delta_{ij}$.
    Let $\rho := u_1 + \dots + u_r$.
    Then, $\overline{Q^\lambda}$ is contained in the half-space
    \[S := \{x \in E \mid (\rho | x) \le (\rho | \lambda)\}.\]
    Note that 
    \[\overline{C_+} = \Bigl\{ \sum_{i \in [r]} a_i u_i \Bigm | a_i \ge 0 \Bigr\},\] 
    and by \cite[Exercise 13.7]{humphreys1972}, we have $(u_i | u_j) \ge 0$ for all $i, j = 1, \dots, r$. Therefore, for a point $x = \sum a_i u_i \in \overline{C_+}$, the inequality $( \rho | x) \le (\rho | \lambda)$ 
    implies $ a_i (u_i| u_i) \le (\rho | \lambda)$. Hence, we have
    \[\overline{C_+} \cap \overline{Q^\lambda} \subseteq \overline{C_+} \cap S \subseteq \Bigl\{ \sum_{i \in [r]} a_i u_i \Bigm | A \ge a_i \ge 0 \Bigr\}\]
    for some $A \in \mathbb{R}_+$. Note that the rightmost set is bounded.
\end{proof}

\begin{remark}
    The assumption $(v_i | v_j) \le 0$ in Theorem \ref{thm-gen} 
    cannot be dropped. For instance, the $P^\lambda$ in Figure \ref{fig-dim2-noncube} is a 2-dimensional polytope not combinatorially equivalent to the square. 
\end{remark}

\begin{figure}[ht]
    \centering
    \begin{tikzpicture}
        \draw [->, gray] (0,0) -- (4,0);
        \draw [->, gray] (0,0) -- (120:2);
        \draw [->] (0,0) -- (0,1);
        \draw [->] (0,0) -- (30:1);
        \coordinate [label=above:$v_1$] (v1) at (0,1); 
        \coordinate [label=above:$v_2$] (v2) at (30:1);
        \coordinate [label=above:$\lambda$] (lambda) at (3,1);
        \fill [gray, opacity=0.3] (lambda) -- (3,0) -- ($(lambda) + (210:2)$) -- (lambda);
        \draw [->, gray] (lambda) -- (3, -0.7);
        \draw [->, gray] (lambda) -- ($(lambda) + (210: 3)$); 
        \coordinate [label = above:$C_+$] (Cplus) at (2,1);
        \coordinate [label = below:$Q^\lambda$] (Qlambda) at (2,-0.2);
        \draw [thick] (lambda) -- (3,0) -- ($(lambda) + (210:2)$) -- (lambda);
        \coordinate [label = above:$P^\lambda$] (P) at (2.5,0.1);
        \coordinate [label=below:$0$] (o) at (0,0);
    \end{tikzpicture}
    \caption{In the Euclidean plane, let $v_1 = (0,1)$ and $v_2 = (\frac{\sqrt{3}}{2}, \frac{1}{2})$. Then, $(v_1|v_2) > 0$. Let $\lambda = (3,1)$. The shaded area is the polytope $P^\lambda := \overline{C_+} \cap \overline{Q^\lambda}$.}
    \label{fig-dim2-noncube}
\end{figure}
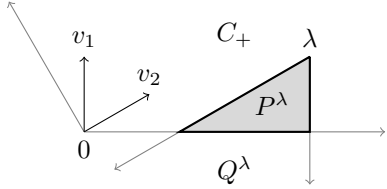

\section{An application to Peterson varieties in classical types} \label{Sec-Application}
Let $G$ be a connected, simply connected, semisimple algebraic group over $\mathbb{C}$, $B$ a Borel subgroup of $G$, and $T \subset B$ a maximal torus. Let $\mathfrak{g}, \mathfrak{b}$, and $\mathfrak{t}$ be the Lie algebras of $G, B$, and $T$, respectively. Let $\mathfrak{g}=\mathfrak{t} \oplus \bigoplus_{\alpha \in \Phi} \mathfrak{g}_\alpha$ be the corresponding root space decomposition, where $\Phi$ is the set of roots. Let $\Delta \subset \Phi$ be the set of simple roots. For each $\alpha \in \Delta$, we choose a non-zero element $e_\alpha \in \mathfrak{g}_{\alpha}$. 
Let $e \in \mathfrak{g}$ be the regular nilpotent element given by 
\begin{equation*}
e:=\sum_{\alpha \in \Delta} e_\alpha \in \mathfrak{g}.
\end{equation*}
The \emph{Peterson variety} $Y \subseteq G/B$, defined by 
\begin{equation*} 
Y:=\Bigl\{g B \in G / B \Bigm| \operatorname{Ad}_{g^{-1}} e \in \mathfrak{b} \oplus \bigoplus_{\alpha \in \Delta} \mathfrak{g}_{-\alpha}\Bigr\},
\end{equation*}
is a remarkable subvariety of the flag variety $G / B$ introduced by Peterson \cite{peterson1997quantum} to study the quantum cohomology ring of the partial flag varieties $G^{\vee}  / P^{\vee}$ for all parabolic subgroups $P^{\vee} \subseteq G^{\vee}$ containing $B^{\vee}$ ($^{\vee}$ means the Langlands dual). It is known that $Y$ is an irreducible projective variety with $\operatorname{dim}_{\mathbb{C}} Y=r$, where $r$ is the rank of $G$ \cite{peterson1997quantum}. 

Let $H^*(Y ; \mathbb{Q})$ be the singular cohomology ring of the Peterson variety $Y$ with rational coefficients. Let $b_i:=\operatorname{dim}_{\mathbb{Q}} H^i(Y; \mathbb{Q})$ be the Betti numbers and $h_{Y}(q):=\sum_{i} b_{i} q^{i/2}$ be the Poincar\'e polynomial. By virtue of the bridge provided by recent results in \cite{horiguchi2021toric}, we use Theorem \ref{cube} to get formulas for the Betti numbers and the Poincar\'e polynomial of the Peterson variety in classical types, which already appear in \cite[p. 199]{brion2004equivariant} and \cite[Equation (2.3)]{harada2015equivariant}. 

\begin{theorem}\label{thm-peterson-betti}
   Let $G$ be a connected, simply-connected, classical type simple algebraic group of rank $r$ over $\mathbb{C}$. Let $Y \subseteq G / B$ be the corresponding Peterson variety. Then $h_{Y}(q)=(1+q)^r$,
    that is, $b_{2i+1}=0$ and $b_{2i}=\binom{r}{i}$ is the binomial coefficient.
\end{theorem}

Horiguchi--Masuda--Shareshian--Song proved in \cite[Corollary 6.3]{horiguchi2021toric} that as graded rings, 
\begin{equation} \label{Peterson=toric}
H^*(Y ; \mathbb{Q}) \cong H^*(X(P^\lambda) ; \mathbb{Q}),
\end{equation}
where $\lambda$ is some regular integral weight, and $X(P^\lambda)$ is the toric variety associated with the normal fan of the dominant weight polytope $P^\lambda$---which is their ``partitioned weight polytope'' $P_{\Phi}(K)$ for $K=[r]$, where $\Phi$ is the root system of $G$. Although their $P_{\Phi}([r])$ is the intersection of the weight polytope with the anti-dominant chamber, this difference is not important since it does not affect our arguments. 

\begin{proof}[Proof of Theorem \ref{thm-peterson-betti}] Since $P^\lambda$ is a simple rational (with respect to the integral weight lattice, see Corollary \ref{cor-vertices}) polytope, it is well-known that the corresponding toric variety $X(P^\lambda)$ is an orbifold with only finite quotient singularities, hence 
\begin{equation} \label{H=IH}
 H^i(X(P^\lambda) ; \mathbb{Q}) \cong IH^i(X(P^\lambda) ; \mathbb{Q}),
\end{equation}
where $IH^i(X(P^\lambda) ; \mathbb{Q})$ is the rational middle-perversity cohomology group of $X(P^\lambda)$ as in \cite{maxim2019intersection}. 
Furthermore, $IH^{2i+1}(X(P^\lambda); \mathbb{Q})=0$, and the Poincar\'e polynomial of $IH^i(X(P^\lambda) ; \mathbb{Q})$ is completely determined by the face numbers of $P^\lambda$ via 
\begin{equation} \label{h-vector}
\sum_{i} \operatorname{dim}_{\mathbb{Q}} IH^{2i}(X(P^\lambda) ; \mathbb{Q}) \cdot q^{i} = \sum_{i=0}^r f_{i}(q-1)^{i},
\end{equation}
where $f_{i}$ is the number of the $i$-dimensional faces of $P^\lambda$ as in \cite{stanley1980number}. 

Since $P^\lambda$ is combinatorially equivalent to the $r$-dimensional cube, we have $f_{i}=\binom{r}{i}2^{r-i}$.
From this, together with \eqref{Peterson=toric}, \eqref{H=IH}, \eqref{h-vector}---or the fact that the toric variety corresponding to the standard $r$-dimensional cube $[0,1]^r$ is $\left(\mathbb{C P}^1\right)^r$---we get $h_{Y}(q)=(1+q)^r$, as desired.
\end{proof}

Rietsch studied the totally non-negative part $Y_{\geq 0}:=Y\cap (SL_{r+1}/B)_{\geq 0}$ of $Y$ \cite{Rietsch2006}, where $(SL_{r+1}/B)_{\geq 0}$ is the totally non-negative part of the flag variety in type $A$, as defined by Lusztig. 
Using ``mirror constructions'', she obtained an elementary proof---without relying on quantum cohomology positivity statements---that $Y_{\geq 0}$ has a cell decomposition given by intersections with open Richardson varieties \cite[Theorem 10.2]{Rietsch2006}. 
She also conjectured that $Y_{\geq 0}$ is homeomorphic to the cube $[0,1]^{r}$ as a cell decomposed space \cite[Conjecture 10.3]{Rietsch2006}. 
In \cite{abe2023totally}, Abe and Zeng constructed a combinatorial cube (which is different from the dominant weight polytope) closely related to the Peterson variety in type $A$ in order to prove Rietsch’s cell decomposition conjecture. 
They also constructed in \cite{abe2023peterson} an explicit morphism from the Peterson variety to the toric orbifold arising from the corresponding root system in any Lie type, and this morphism induces a ring isomorphism between their rational cohomology rings. 
Based on their results and the results in this paper, we pose the following question.

\begin{question}
Is there a cell decomposition of the totally non-negative part $Y_{\geq 0}$ of the Peterson variety in any Lie type, such that $Y_{\geq 0}$ is homeomorphic to the cube $[0,1]^{\operatorname{dim}_{\mathbb{C}} Y}$ as a cell decomposed space?
\end{question}

If the answer to the above question is yes, our main theorem could be seen as a ``combinatorial shadow'' of the cell decomposition of $Y_{\geq 0}$.

\bibliographystyle{amsplain}
\bibliography{template}

\end{document}